\documentclass[a4paper,11pt]{article}
\usepackage{amsmath}
\usepackage{amssymb}
\usepackage{amsthm}
\usepackage{a4wide}
\usepackage{epsfig}
\usepackage{subfigure}
\usepackage{url}
\usepackage{xcolor, soul}
\newtheorem{algorithm}{Algorithm}[section]
\newtheorem{theorem}{Theorem}[section]
\newtheorem{lemma}{Lemma}[section]
\newcommand{\solvername}{\textrm{MRS$^3$}}
\newcommand{\algospace}{\vspace{-0.4cm}}
\newcommand{\tud}{Delft University of Technology,
Faculty of Electrical Engineering, Mathematics and Computer Science,
Delft Institute of Applied Mathematics,
Mekelweg 4, 2628 CD Delft, The Netherlands,
\textbf{e-mail: c.vuik@tudelft.nl}}
\title{A comparison of Krylov methods \\ for Shifted Skew-Symmetric Systems}
\author{R. Idema \and C. Vuik\thanks{\tud}}
\date{\ }
\begin{document}
\maketitle
\setlength{\parindent}{0pt}
\setlength{\parskip}{1.5ex}
\begin{abstract}
It is well known that for general linear systems, only optimal Krylov methods with long recurrences exist.
For special classes of linear systems it is possible to find optimal Krylov methods with short recurrences.
In this paper we consider the important class of linear systems with a shifted skew-symmetric coefficient matrix.
We present the {\solvername} solver, a minimal residual method
that solves these problems using short vector recurrences.
We give an overview of existing Krylov solvers that can be used to solve these problems,
and compare them with the {\solvername} method, both theoretically and by numerical experiments.
From this comparison we argue that the {\solvername} solver is the fastest and most robust
of these Krylov method for systems with a shifted skew-symmetric coefficient matrix.
\end{abstract}
{\bf Keywords:} Lanczos, Krylov, Minimal Residual, Short Recurrences, Shifted Skew-Symmetric
\\[2ex]
{\bf AMS Subject Classification:} 65F10
\section{Introduction}\label{sec:introduction}

In this paper we explore Krylov subspace methods that can solve systems of linear equations of the form
\begin{equation}\label{eq:sss_system}
A\mathbf{x} = \mathbf{b},
\end{equation}
where $A \in \mathbb{R}^{n \times n}$ is a shifted skew-symmetric matrix, i.e.,
\begin{equation}\label{eq:sss_matrix}
A = \alpha I + S,~ \alpha\in\mathbb{R},~ S^T = -S.
\end{equation}
Throughout this paper we will use $I$ for the identity matrix of appropriate size,
$H$ for symmetric matrices, and $S$ for skew-symmetric matrices as above.
Further we will use the abbreviation SSS for shifted skew-symmetric.
Note that SSS matrices are normal, i.e., $AA^T = A^TA$.

Our research on this problem was previously available as a technical report~\cite{idema_vuik:minres_sss}.
Due to the increasing interest in shifted skew-symmetric problems, we decided to now formally publish the work.


Shifted skew-symmetric systems arise in many scientific and engineering applications,
like computational fluid dynamics, linear programming and systems theory.

In computational fluid dynamics, SSS systems arise when dealing with
Navier-Stokes equations with a large~\cite{golub_vanderstraeten:precon_skewsymmsplit}
or a small~\cite{Gol98W} Reynolds number (see also \cite{Bai03GN}).

Consider $B$ to be a large nonsingular matrix,
which is a discrete version of an advection-diffusion problem.
The Hermitian splitting can be used to decompose $B$
in its symmetric part $H$ and its skew-symmetric part $S$:
\[ B = H+S, \textrm{ where } H=\frac{B+B^T}{2} \textrm{ and } S=\frac{B-B^T}{2}. \]

If the diffusion is important, i.e., if the Reynolds number is small,
and if the symmetric part $H$ of $B$ is positive definite,
$H^{-1}$ can be used as a preconditioner to solve a system $B\mathbf{x} = \mathbf{b}$.
Note that to compute $\mathbf{v} = H^{-1} \mathbf{w}$ efficiently, multigrid can be used.
The preconditioning can be done as follows: 
\begin{equation}\label{eq:precond_to_sss}
H^{-\frac12} B H^{-\frac12} \mathbf{y} = H^{-\frac12} \mathbf{b},
\textrm{ where } \mathbf{x} = H^{-\frac12} \mathbf{y}.
\end{equation}
This equation can then be rewritten as
\[ (I + H^{-\frac12} S H^{-\frac12})\mathbf{y} = H^{-\frac12}\mathbf{b}, \]
which is an SSS system (compare~\cite{Gol98W}).

On the other hand if advection is dominant, i.e., if the Reynolds number is large,
\[\left(I-(\alpha I+S)^{-1}(H-\alpha I) \right)(\alpha I+S)^{-1}\]
can be used as a preconditioner (see~\cite{golub_vanderstraeten:precon_skewsymmsplit} eq.~(1.7) and~(3.1)).
Applying this preconditioner to a vector $\mathbf{w}$ implies that SSS systems of the form 
$(\alpha I+S)\mathbf{v} = \mathbf{w}$ have to be solved.

For an application in linear programming consider interior point methods,
a popular way of solving linear programming problems.
When solving a linear program with such a method, using a self-dual embedding of the
problem takes slightly more computational time per iteration
but has several important advantages such as having a centered starting point and
detecting infeasibility by convergence, as described in~\cite{roos:linopt_theoryandalgorithms}.
Therefore, most modern solvers use such an embedding.

Interior point methods are iterative schemes that search
for an optimal solution from within the strictly feasible set.
In each iteration a step $\Delta\mathbf{x}_i$ to add to the current solution is generated.
To calculate this step, a large sparse system has to be solved, that is of the form
\begin{equation}\label{sys:linopt_solveA}
(D_i+S)\Delta\mathbf{x}_i = \mathbf{b}_i,
\end{equation}
where $S$ is a skew-symmetric matrix and $D_i$ is a diagonal matrix
with strictly positive diagonal entries.
Using the same preconditioning as in equation~(\ref{eq:precond_to_sss}),
we can rewrite system~(\ref{sys:linopt_solveA}) as
\begin{equation}\label{sys:linopt_solveB}
(I+D_i^{-\frac12}SD_i^{-\frac12})\mathbf{y}_i = D_i^{-\frac12}\mathbf{b}_i,
\textrm{ where } \mathbf{y}_i = D_i^{\frac12}\Delta\mathbf{x}_i \,.
\end{equation}
This again is an SSS system.
Note that the preconditioning used is in fact diagonal scaling.

In systems theory, shifted skew-symmetric linear systems arise in the discretization of
port-Hamiltonian problems~\cite{beattie_mehrmann_xu_zwart:linear_port_hamiltonian}.


In this paper, we aspire to give an overview of Krylov methods available to solve Shifted Skew-Symmetric systems,
and we present {\solvername}, a Minimal Residual method for SSS systems, designed specifically for solving such systems.
In Section~\ref{sec:overviewexisting} we review existing methods, and their application to SSS systems.
The {\solvername} algorithm is presented in Section~\ref{sec:idemasolver}.
In Section~\ref{sec:solver_comparison} we do a theoretical comparison of the treated methods,
followed in Section~\ref{sec:numerical_experiments} by the results of our numerical experiments.
Finally, in Section~\ref{sec:conclusions} we present our conclusions with respect to solvers for SSS systems.

\section{Overview of existing methods}\label{sec:overviewexisting}

In this section we give an overview of some existing Krylov subspace methods,
that can be used to solve shifted skew-symmetric systems.
In iteration $j$, a Krylov subspace method approximates the solution
with $\mathbf{x}_j=\mathbf{x}_0+\mathbf{s}_j$, where $\mathbf{s}_j\in\mathcal{K}_j(A,\mathbf{r}_0)$.
Here $\mathbf{x}_0$ is the initial solution, $A$ is the coefficient matrix of the system,
$\mathbf{r}_0 = \mathbf{b}-A\mathbf{x}_0$ is the initial residual,
and $\mathcal{K}_j(A,\mathbf{r}_0)$ is the Krylov subspace:
\[ \mathcal{K}_j(A,\mathbf{r}_0) = \textrm{span}\{\mathbf{r}_0, A\mathbf{r}_0, \ldots, A^{j-1}\mathbf{r}_0\}. \]
Since in every iteration the Krylov subspace is expanded, a new approximation
within the larger subspace can be generated that is never worse than the previous one.

Two important properties of Krylov subspace methods are optimality and short recurrences.
An algorithm has the optimality property if the generated approximation for the solution is,
measured in some norm, the best within the current Krylov subspace.
The short recurrences property is satisfied if the algorithm can generate
the next approximation using only data from the last few iterations.

For general coefficient matrices the above properties cannot be attained simultaneously.
However methods satisfying both properties do exist for matrices of the form
\[ A = e^{i\theta}(\sigma I+T) \textrm{, where } \theta\in\mathbb{R},~\sigma\in\mathbb{C},~T^H = T. \]
These results are due to Voevodin~\cite{voevodin:nsa_gen_conj_grad} and
Faber and Manteuffel~\cite{faber_manteuffel:existence_cg},~\cite{faber_manteuffel:orth_error_meth}.
Taking $\theta=\pi/2$, $\sigma=-i\alpha$ with $\alpha \in \mathbb{R}$
and $T=-iS$ with $S^T=-S$, we get an SSS matrix as given in equation~(\ref{eq:sss_matrix}).
This implies that a Krylov subspace method for SSS systems exists,
that has both the optimality property and short recurrences.

\subsection{General methods}\label{subsec:generalmethods}

An SSS system can be solved with any solver for general systems of linear equations.
We will treat a few widely used methods, that nicely illustrate
the findings of Voevodin and Faber and Manteuffel mentioned above.

GMRES~\cite{saad_schultz:gmres} generates optimal approximations to the solution,
but needs vectors from all the previous iterations to do so.
For more details see the remarks about GMRES at the start of Section~\ref{sec:solver_comparison}.

GCR~\cite{eisenstat_elman_schultz:var_iter_meth_nonsymm},~\cite{vdvorst_vuik:gmresr}
also generates optimal approximations, and generally needs vectors from all previous iterations.
However for SSS systems, the orthogonalization can be done with information from the last iteration only.
Truncating the orthogonalization of GCR is commonly known as Truncated GCR,
or Orthomin($k$)~\cite{vinsome:orthomin}.
So, in other words, for SSS systems Orthomin(1) is the same as full GCR,
as is shown in Theorem~\ref{th:gcr_is_orthomin}.
However, there are examples where GCR breaks down, whereas GMRES does not.
For details see Section~\ref{sec:trunc_gcr} below.

Bi-CGSTAB \cite{vdvorst:bicgstab} uses short recurrences but does not have the optimality property.
Usually it converges fast but it is not very robust.

Finally, CGNR \cite{paige_saunders:lsqr} solves the normal equations
$A^TA\mathbf{x}=A^T\mathbf{b}$ with the CG method.
This solver achieves both optimality and short recurrences,
but in a different Krylov subspace, namely $\mathcal{K}_j(A^TA,A^T\mathbf{r}_0)$.
Since the condition number is squared when working with $A^TA$,
convergence can be very slow for ill-conditioned systems.
This method is used for solving SSS systems by Golub and Vanderstraeten
in their treatment of the preconditioning of matrices with a
large skew-symmetric part~\cite{golub_vanderstraeten:precon_skewsymmsplit}.
For a regularization technique using CGNR,
designed specifically to deal with very ill-conditioned skew-symmetric systems,
we refer to~\cite{chen_shen:regularized_cgnr}.

\subsection{Truncated GCR}\label{sec:trunc_gcr}

Below we present the GCR algorithm, truncated after one orthogonalization step,
also known as Orthomin(1).

\begin{figure}[ht]
\hrulefill
\vspace{-.2cm}
\begin{algorithm}[Truncated GCR]\label{algo:truncgcr}~\\
\algospace
\begin{list}{}{}
\item Let $\mathbf{x}_0$ be given, $\mathbf{r}_0 = \mathbf{b}-A\mathbf{x}_0$,
$\mathbf{v}_0=\mathbf{0}$, $j=0$
\item While \emph{not converged} do

    \begin{list}{}{}
    \item $j = j+1$
    \item $\mathbf{s}_j = \mathbf{r}_{j-1}$
      and $\mathbf{v}_j = A\mathbf{r}_{j-1}$

    \item $\mu_j = \left(\mathbf{v}_{j-1},\mathbf{v}_j\right)$
    \item $\mathbf{s}_j = \mathbf{s}_j - \mu_j\mathbf{s}_{j-1}$
      and $\mathbf{v}_j = \mathbf{v}_j - \mu_j\mathbf{v}_{j-1}$

    \item $\beta_j = \left|\left|\mathbf{v}_j\right|\right|_2$
    \item $\mathbf{s}_j = \mathbf{s}_j / \beta_j$
      and $\mathbf{v}_j = \mathbf{v}_j / \beta_j$

    \item $\gamma_j = \left(\mathbf{r}_{j-1},\mathbf{v}_j\right)$
    \item $\mathbf{x}_j = \mathbf{x}_{j-1} + \gamma_j\mathbf{s}_j$
      and $\mathbf{r}_j = \mathbf{r}_{j-1} - \gamma_j\mathbf{v}_j$
    \end{list}

\item Endwhile
\end{list}
\end{algorithm}
\vspace{-.6cm}\hrulefill\vspace{.2cm}
\end{figure}

The full GCR algorithm is the same, except that the orthogonalization step reads
\begin{equation}\label{eq:gcr_full_orthogonalization}
\mathbf{s}_j
= \mathbf{s}_j - \sum_{i=1}^{j-1}
\left(\mathbf{v}_i,\mathbf{v}_j\right)\mathbf{s}_i \textrm{~~and~~}
\mathbf{v}_j
= \mathbf{v}_j - \sum_{i=1}^{j-1}
\left(\mathbf{v}_i,\mathbf{v}_j\right)\mathbf{v}_i.
\end{equation}
Thus the value of $\mathbf{v}_j$ at the end of each iteration is
$\mathbf{v}_j = \frac{1}{\beta_j}\left(A\mathbf{r}_{j-1}
- \sum_{i=1}^{j-1}\left(\mathbf{v}_i,A\mathbf{r}_{j-1}\right)\mathbf{v}_i\right)$,
for full GCR.
We will use $\beta_j\mathbf{v}_j$
to denote the value of $\mathbf{v}_j$ just before the normalization step.

First note that in full GCR $\mathbf{v}_j\perp\mathbf{v}_1,\ldots,\mathbf{v}_{j-1}$ by construction, thus
\begin{equation}\label{eq:gcr_vv_orthogonal}
\left(\mathbf{v}_j,\mathbf{v}_i\right)=0, \textrm{ for all } i<j.
\end{equation}
Further note that
\begin{equation}\label{eq:gcr_rv_orthogonal}
\left(\mathbf{r}_j,\mathbf{v}_i\right)=0, \textrm{ for all } i \leq j,
\end{equation}
because we have
$\left(\mathbf{r}_j,\mathbf{v}_j\right)
=\left(\mathbf{r}_{j-1}-\gamma_j\mathbf{v}_j,\mathbf{v}_j\right)
=\left(\mathbf{r}_{j-1},\mathbf{v}_j\right)-\gamma_j\left(\mathbf{v}_j,\mathbf{v}_j\right)
=\gamma_j-\gamma_j=0$,
and for $i<j$ we find
$\left(\mathbf{r}_j,\mathbf{v}_i\right)
=\left(\mathbf{r}_i-\sum_{k=i+1}^j\gamma_k\mathbf{v}_k,\mathbf{v}_i\right)
=\left(\mathbf{r}_i,\mathbf{v}_i\right)-\sum_{k=i+1}^j\gamma_k\left(\mathbf{v}_k,\mathbf{v}_i\right)=0.$

Below we will present a few properties of non-truncated GCR,
that are often regarded common knowledge in the linear algebra community,
but that we have not been able to find references to their explicit proofs for.
For completeness we have therefore included proofs ourselves.

\begin{lemma}\label{eq:gcr_breakdown}
In GCR, if $\gamma_j=0$ for some $j$, the algorithm breaks down in iteration $j+1$.
\end{lemma}
\begin{proof}
If $\gamma_j=0$ then $\mathbf{r}_j=\mathbf{r}_{j-1}$.
Therefore we can write
\[ \beta_{j+1}\mathbf{v}_{j+1}
= A\mathbf{r}_j - \sum_{i=1}^j \left(\mathbf{v}_i,A\mathbf{r}_j\right)\mathbf{v}_i \\
= A\mathbf{r}_{j-1} - \sum_{i=1}^{j-1} \left(\mathbf{v}_i,A\mathbf{r}_{j-1}\right)\mathbf{v}_i
- \left(\mathbf{v}_j,A\mathbf{r}_{j-1}\right)\mathbf{v}_j. \]
The first two terms of the right hand side together are equal to $\beta_j\mathbf{v}_j$,
whereas the last term we can rewrite using that
$A\mathbf{r}_{j-1} = \beta_j\mathbf{v}_j
+ \sum_{i=1}^{j-1}\left(\mathbf{v}_i,A\mathbf{r}_{j-1}\right)\mathbf{v}_i$,
and the orthogonality relation~(\ref{eq:gcr_vv_orthogonal}):
\[ \left(\mathbf{v}_j,A\mathbf{r}_{j-1}\right)\mathbf{v}_j
= \left(\mathbf{v}_j,\beta_j\mathbf{v}_j
+ \sum_{i=1}^{j-1}\left(\mathbf{v}_i,A\mathbf{r}_{j-1}\right)\mathbf{v}_i\right)\mathbf{v}_j
= \beta_j\left(\mathbf{v}_j,\mathbf{v}_j\right)\mathbf{v}_j
= \beta_j\mathbf{v}_j. \]
Thus we find that $\beta_{j+1}\mathbf{v}_{j+1} = \beta_j\mathbf{v}_j-\beta_j\mathbf{v}_j=0$.
Then the normalization factor $\beta_{j+1}=0$,
and the algorithm breaks down on the calculation of $\mathbf{v}_{j+1}=\frac{\mathbf{0}}{0}$.
\end{proof}

\begin{theorem}\label{th:gcr_breakdown_skewsymm}
When GCR is applied to a system with skew-symmetric coefficient matrix $S$,
the algorithm breaks down in the second iteration.
\end{theorem}
\begin{proof}
Using that $\mathbf{z}^TS\mathbf{z}=0$ for any vector $\mathbf{z}$, we find
$\gamma_1 = \left(\mathbf{r}_0,\mathbf{v}_1\right)
= \left(\mathbf{r}_0,\frac{1}{\beta_1}S\mathbf{r}_0\right) = 0$.
Thus the statement follows readily from Lemma~\ref{eq:gcr_breakdown}.
\end{proof}

As mentioned before, for SSS systems Algorithm~\ref{algo:truncgcr}
gives the same iterates as full GCR.
The following theorem proves this fact,
by showing that the coefficients of all those orthogonalization factors
that are omitted in Algorithm~\ref{algo:truncgcr}, are equal to 0 for SSS systems.

\begin{theorem}\label{th:gcr_is_orthomin}
When GCR is applied to a system with shifted skew-symmetric coefficient matrix $A=\alpha I+S$, then
\[ \left(\mathbf{v}_i,A\mathbf{r}_j\right) = 0, \textrm{ for all~}i<j. \]
\end{theorem}
\begin{proof}
Using that $A=\alpha I+S$ we can write
\begin{equation}\label{eq:gcr_sss_trunc1}
\left(\mathbf{v}_i,A\mathbf{r}_j\right) =
\alpha\left(\mathbf{v}_i,\mathbf{r}_j\right) + \left(\mathbf{v}_i,S\mathbf{r}_j\right) =
-\alpha\left(\mathbf{v}_i,\mathbf{r}_j\right)- \left(S\mathbf{v}_i,\mathbf{r}_j\right) =
-\left(A\mathbf{v}_i,\mathbf{r}_j\right).
\end{equation}
Note that in the second equality we used the fact that
$\left(\mathbf{v}_i,\mathbf{r}_j\right)=0$, see equation~(\ref{eq:gcr_rv_orthogonal}).

Next, rewriting $\mathbf{v}_i$ using the residual update expression
$\mathbf{r}_i = \mathbf{r}_{i-1} - \gamma_i\mathbf{v}_i$, we have
\begin{equation}\label{eq:gcr_sss_trunc2}
\left(A\mathbf{v}_i,\mathbf{r}_j\right) =
\left(A \frac{1}{\gamma_i}\left(\mathbf{r}_{i-1}-\mathbf{r}_i\right), \mathbf{r}_j\right) =
\frac{1}{\gamma_i}\left(A\mathbf{r}_{i-1}, \mathbf{r}_j\right) -
\frac{1}{\gamma_i}\left(A\mathbf{r}_i    , \mathbf{r}_j\right).
\end{equation}
Note that we can assume $\gamma_i\neq 0$,
as otherwise the algorithm would have broken down.

Finally, we will use that
\begin{equation}\label{eq:gcr_sss_trunc3}
\left(A\mathbf{r}_i,\mathbf{r}_j\right)=0, \textrm{ for all } i<j,
\end{equation}
which follows from the fact that using the orthogonalization
formula~(\ref{eq:gcr_full_orthogonalization}) for $\mathbf{v}_j$,
we can write $A\mathbf{r}_i$ as a linear combination of $\mathbf{v}_1,\ldots,\mathbf{v}_{i+1}$,
which are all orthogonal to $\mathbf{r}_j$ due to relation~(\ref{eq:gcr_rv_orthogonal}).

Combining equations~(\ref{eq:gcr_sss_trunc1}), (\ref{eq:gcr_sss_trunc2}),
and~(\ref{eq:gcr_sss_trunc3}) it follows that
$\left(\mathbf{v}_i,A\mathbf{r}_j\right)=0$ for all $i<j$.
\end{proof}

\subsection{Generalized Conjugate Gradient method}

The Generalized Conjugate Gradient method,
proposed by Concus and Golub~\cite{concus_golub:gencg_nonsymm}
and Widlund~\cite{widlund:lanczos_nonsymm},
is an iterative Lanczos method for solving systems $A\mathbf{x}=\mathbf{b}$
where $A$ has a positive definite symmetric part $H$.
The SSS matrix~(\ref{eq:sss_matrix}) satisfies this requirement if $\alpha>0$,
and for $\alpha<0$ we can easily meet it by solving $-A\mathbf{x}=-\mathbf{b}$.
Thus we can use this method to solve any SSS system with $\alpha\neq 0$.

\begin{figure}[ht]
\hrulefill
\vspace{-.2cm}
\begin{algorithm}[Generalized Conjugate Gradient]\label{algo:gencg}~\\
\algospace
\begin{list}{}{}
\item Let $\mathbf{x}_{-1}=\mathbf{x}_0=0\,,~j=0$
\item While \emph{not converged} do

    \begin{list}{}{}
    \item Solve $H\mathbf{v}_j = \mathbf{b} - A\mathbf{x}_j$
    \item $\rho_j = (H\mathbf{v}_{j},\mathbf{v}_j)$

    \item If $j=0$
        \begin{list}{}{} \item $\omega_j=1$ \end{list}
    \item Else
        \begin{list}{}{} \item $\omega_j=(1+(\rho_j/\rho_{j-1})/\omega_{j-1})^{-1}$ \end{list}
    \item Endif

    \item $\mathbf{x}_{j+1} = \mathbf{x}_{j-1} + \omega_j(\mathbf{v}_j + \mathbf{x}_j - \mathbf{x}_{j-1})$
    \item $j=j+1$
    \end{list}

\item Endwhile
\end{list}
\end{algorithm}
\vspace{-.6cm}\hrulefill\vspace{.2cm}
\end{figure}

The Generalized Conjugate Gradient method does not have the optimality property.
However, it has been proved that the iterates are optimal in some affine subspace
other than the Krylov subspace~\cite{eisenstat:note_genconjgrad}.
In practice this method is rarely used, as it has been superseded by the CGW method by the same authors.
Therefore we will not go into any further details on
the Generalized Conjugate Gradient method in this paper.

\subsection{Concus, Golub, Widlund method}\label{sec:CGW}

Related to the Generalized Conjugate Gradient method presented above,
is the method by Concus, Golub and Widlund (CGW)
described in Section~9.6 of~\cite{saad:itmeth_sparselinsys}.
This method also solves systems of linear equations with
a coefficient matrix with positive definite symmetric part.
But it does so using a two-term recursion, as opposed to
the three-term recursion used by the Generalized Conjugate Gradient method.

Below we present the CGW algorithm.
Therein $H$ is again the symmetric part of $A$.
The algorithm is identical to the preconditioned CG method,
except for the minus sign used in the calculation of
$\beta_j$ (see~\cite{saad:itmeth_sparselinsys}, Section~9.2).
The CGW method can be used to solve
SSS systems~{(\ref{eq:sss_system})},~{(\ref{eq:sss_matrix})} with $\alpha\neq 0$.
Note that in this case $H=\alpha I$, and we can eliminate $\mathbf{z}_j$
by substituting $\mathbf{z}_j = \frac{1}{\alpha}\mathbf{r}_j$,
thus simplifying the algorithm.

\begin{figure}[ht]
\hrulefill
\vspace{-.2cm}
\begin{algorithm}[CGW]\label{algo:cgw} ~\\
\algospace
\begin{list}{}{}
\item Let $\mathbf{x}_0$ be given, $\mathbf{r}_0 = \mathbf{b}-A\mathbf{x}_0$,
Solve $H\mathbf{z}_0 = \mathbf{r}_0$, $\mathbf{p}_0 = \mathbf{z}_0$, $j=0$
\item While \emph{not converged} do

    \begin{list}{}{}
    \item $\alpha_j = (\mathbf{r}_j,\mathbf{z}_j)/(A\mathbf{p}_j,\mathbf{z}_j)$
    \item $\mathbf{x}_{j+1} = \mathbf{x}_j + \alpha_j \mathbf{p}_j$
    \item $\mathbf{r}_{j+1} = \mathbf{r}_j - \alpha_j A \mathbf{p}_j$
    \item Solve $H\mathbf{z}_{j+1} = \mathbf{r}_{j+1}$
    \item $\beta_j = -(\mathbf{z}_{j+1},\mathbf{r}_{j+1})/(\mathbf{z}_j,\mathbf{r}_j)$
    \item $\mathbf{p}_{j+1} = \mathbf{z}_{j+1} + \beta_j \mathbf{p}_j$
    \item $j=j+1$
    \end{list}

\item End while
\end{list}
\end{algorithm}
\vspace{-.6cm}\hrulefill\vspace{.2cm}
\end{figure}

The CGW algorithm uses short recurrences,
but as it is a Galerkin method~(see~\cite{Gut02R} p.~13)
it does not have the optimality property.

\subsection{Huang, Wathen, Li method}

Huang, Wathen and Li~\cite{huang_wathen_li:itmeth_skewsymm} described a method to solve
the SSS system~{(\ref{eq:sss_system})},~{(\ref{eq:sss_matrix})} with $\alpha=0$.
We denote this method by HWL, after the names of the authors.

\begin{figure}[ht]
\hrulefill
\vspace{-.2cm}
\begin{algorithm}[HWL]\label{algo:hwl} ~\\
\algospace
\begin{list}{}{}
\item Let $\mathbf{x}_0$ be given, $\mathbf{r}_0 = \mathbf{b}-A\mathbf{x}_0$,
$\mathbf{p}_0=A\mathbf{r}_0$, $j=0$
\item While \emph{not converged} do

    \begin{list}{}{}
    \item $\alpha_j = (\mathbf{r}_j,A\mathbf{p}_j)/(A\mathbf{p}_j,A\mathbf{p}_j)$
    \item $\mathbf{x}_{j+1} = \mathbf{x}_j + \alpha_j \mathbf{p}_j$
    \item $\mathbf{r}_{j+1} = \mathbf{b} - A \mathbf{x}_{j+1}$
    \item $\beta_j = (A^2\mathbf{p}_j,A\mathbf{r}_{j+1})/(A\mathbf{p}_j,A\mathbf{p}_j)$
    \item $\mathbf{p}_{j+1} = A\mathbf{r}_{j+1} + \beta_j \mathbf{p}_j$
    \item $j=j+1$
    \end{list}

\item End while
\end{list}
\end{algorithm}
\vspace{-.6cm}\hrulefill\vspace{.2cm}
\end{figure}

In exact arithmetic the HWL algorithm actually generates
the same approximations to the solution as the CGNR method,
as proved in~\cite{idema_vuik:minres_sss}.
As we already treat the CGNR algorithm in this paper,
we will not go into further detail on the HWL method.

\section{{\solvername} solver}\label{sec:idemasolver}

In the previous section we described various existing methods that can be used
to solve SSS systems~(\ref{eq:sss_system}),~(\ref{eq:sss_matrix}).
Each of these methods has its own drawback.
The general methods do not achieve both short recurrences and optimality,
while the specialized methods do not work for all values of $\alpha$.

In this section we present a solver for SSS systems that satisfies both the
short recurrences and the optimality property, and can be used for all values of $\alpha\in\mathbb{R}$.
This Minimal Residual method for Shifted Skew-Symmetric systems, or {\solvername},
is a Krylov subspace method that is based on the Lanczos algorithm~\cite{lanczos:itmeth_eigenvalue}.

\subsection{Shifted skew-symmetric Lanczos algorithm}\label{sec:idemasolver_lanczos_nonsymmetric}

For SSS matrices~(\ref{eq:sss_matrix}),
the non-symmetric Lanczos algorithm can be reduced to Algorithm~\ref{algo:lanczos_sss} below.
For details see~\cite{idema_vuik:minres_sss}.

\begin{figure}[ht]
\hrulefill
\vspace{-.2cm}
\begin{algorithm}[Shifted skew-symmetric Lanczos algorithm]\label{algo:lanczos_sss} ~\\
\algospace
\begin{list}{}{}
\item Let $\mathbf{q}_0=0\,,~j=0$ \\
Choose $\mathbf{p}_1$ with $\mathbf{p}_1 \neq \mathbf{0}$ and let $\beta_1=||\mathbf{p}_1||_2$
\item While $\beta_{j+1} > 0$ do

    \begin{list}{}{}
    \item $j = j+1$
    \item $\mathbf{q}_j = -\mathbf{p}_j / \beta_j$
    \item $\mathbf{p}_{j+1} = S \mathbf{q}_j - \beta_j \mathbf{q}_{j-1}$
    \item $\beta_{j+1} = ||\mathbf{p}_{j+1}||_2$
    \end{list}

\item End while
\end{list}
\end{algorithm}
\vspace{-.6cm}\hrulefill\vspace{.2cm}
\end{figure}

Besides the obvious fact that the computational work is greatly reduced
with respect to the non-symmetric Lanczos algorithm,
the SSS Lanczos algorithm also has the nice property that
serious breakdown will (in exact arithmetic) not occur,
as $\beta_j = ||\mathbf{p}_{j+1}||_2 = 0 \Leftrightarrow \mathbf{p}_{j+1} = 0$.

Defining $Q=\left[\mathbf{q}_1\ldots\mathbf{q}_j\right]$ the following relation holds:
\begin{equation}\label{eq:lanczos_sss_recurrences_ext}
A Q_j = Q_{j+1} \tilde{T}_j ,
\end{equation}
where the $(j\!+\!1)\!\times\!j$ extended Ritz matrix $\tilde{T}_j$ is defined as
\[ \tilde{T}_j =
\left[ \begin{array}{ccccc}
    \alpha   & \beta_2  & 0       & \cdots   & 0\\
    -\beta_2 & \alpha   & \beta_3 & \ddots   & \vdots \\
    0        & -\beta_3 & \ddots  & \ddots   & 0 \\
    \vdots   & \ddots   & \ddots  & \ddots   & \beta_j \\
    0        & \cdots   & 0       & -\beta_j & \alpha \\
    0        & \cdots   & \cdots  & 0        & -\beta_{j+1} \\
\end{array} \right]. \]

Note that the same result can be obtained by applying
the Arnoldi method~\cite{arnoldi:min_iter_eigenvalue_problem} to SSS matrices.
Like the Arnoldi algorithm reduces to the Lanczos method for symmetric matrices,
it also reduces to the shifted skew-symmetric Lanczos algorithm for SSS matrices
(see also~\cite{huckle:arnoldi_normal_matrices}).
This method was used by Jiang~{\cite{jiang:solving_sss_systems}} to derive a method
that is equivalent to the {\solvername} algorithm derived here.

\subsection{Solving shifted skew-symmetric systems}\label{sec:idemasolver_solving}

Krylov subspace methods can be categorized by the way the approximation
$\mathbf{x}_j=\mathbf{x}_0+\mathbf{s}_j$ of the solution $\mathbf{x}$ is calculated.
Minimal residual methods choose $\mathbf{s}_j$ such that
the norm of the residual $\mathbf{r}_j$ is minimized.
Orthogonal residual (or Galerkin) methods calculate $\mathbf{s}_j$ such that $Q^T_j\mathbf{r}_j=0$.
We will follow the minimal residual path, because it satisfies the
optimality property described in Section~\ref{sec:overviewexisting},
whereas orthogonal residual methods generally do not.

We start the solver with an initial guess $\mathbf{x}_0$.
Then in each iteration $j$ we will calculate $\mathbf{s}_j \in \mathcal{K}_j(A,\mathbf{r}_0)$
such that $\left|\left|\mathbf{r}_j\right|\right|_2$ is minimized.
The vectors generated by the above derived shifted skew-symmetric
Lanczos algorithm~\ref{algo:lanczos_sss} will be used to rewrite
$||\mathbf{r}_j||_2$ to such a form that we can calculate $\mathbf{s}_j$ from it.

We start Algorithm~\ref{algo:lanczos_sss} with $\mathbf{p}_1=\mathbf{r}_0=\mathbf{b}-A\mathbf{x}_0$.
Then, since the columns of $Q_j$ form a basis for $\mathcal{K}_j(A,\mathbf{q}_1)$ and
\begin{equation}\label{eq:q1}
\mathbf{q}_1 = -\frac{\mathbf{p}_1}{\left|\left|\mathbf{p}_1\right|\right|_2}
= -\frac{\mathbf{r}_0}{\left|\left|\mathbf{r}_0\right|\right|_2},
\end{equation}
the columns of $Q_j$ also form a basis for the Krylov subspace $\mathcal{K}_j(A,\mathbf{r}_0)$.
Therefore, for all $\mathbf{s}_j\in\mathcal{K}^j(A,\mathbf{r}_0)$ 
there exists a $\boldsymbol{\xi}_j \in \mathbb{R}^j$ such that
$\mathbf{s}_j=Q_j\boldsymbol{\xi}_j$, and we can write
\[ \left|\left|\mathbf{r}_j\right|\right|_2 = \left|\left|\mathbf{b}-A\mathbf{x}_j\right|\right|_2
= \left|\left|\mathbf{b}-A\mathbf{x}_0-A\mathbf{s}_j\right|\right|_2
= \left|\left|\mathbf{r}_0-AQ_j\boldsymbol{\xi}_j\right|\right|_2 . \]
Now, using equations~(\ref{eq:lanczos_sss_recurrences_ext}) and~(\ref{eq:q1}) we get
\[ \left|\left|\mathbf{r}_j\right|\right|_2
= \left|\left|\mathbf{r}_0-Q_{j+1}\tilde{T}_j\boldsymbol{\xi}_j\right|\right|_2
= \left|\left|Q_{j+1}(-\left|\left|\mathbf{r}_0\right|\right|_2\mathbf{e}_1
-\tilde{T}_j\boldsymbol{\xi}_j)\right|\right|_2 . \]
Since the matrix $Q_{j+1}$ is orthogonal,
and the 2-norm is invariant with respect to orthogonal transformations,
it follows that
\begin{equation} \label{eq:res_in_xi}
\left|\left|\mathbf{r}_j\right|\right|_2
= \left|\left| (\left|\left|\mathbf{r}_0\right|\right|_2\mathbf{e}_1
+ \tilde{T}_j\boldsymbol{\xi}_j) \right|\right|_2 .
\end{equation}
A minimal residual is therefore obtained by choosing
$\boldsymbol{\xi}_j = \boldsymbol{\hat\xi}_j$, where
\begin{equation}\label{eq:minres_xi}
\boldsymbol{\hat\xi}_j = \arg\min_{\boldsymbol{\xi}_j \in \mathbb{R}^j}
\left|\left| (\left|\left|\mathbf{r}_0\right|\right|_2\mathbf{e}_1
+ \tilde{T}_j\boldsymbol{\xi}_j) \right|\right|_2,
\end{equation}
i.e., $\boldsymbol{\hat\xi}_j$ is the least-squares solution of the linear system
\begin{equation}\label{eq:minres_system}
\tilde{T}_j\boldsymbol{\xi}_j = -\left|\left|\mathbf{r}_0\right|\right|_2\mathbf{e}_1 \,.
\end{equation}

This least-squares solution can be found with the help of Givens rotations.
A Givens rotation of a vector is the multiplication of that vector by a square orthogonal matrix of the form
\[ G_\mathbf{y}\left(k,l\right) =
\left[ \begin{array}{ccccc}
    I      & 0      & \cdots & \cdots & 0 \\
    0      & c      & \ddots & s      & \vdots \\
    \vdots & \ddots & I      & \ddots & \vdots \\
    \vdots & -s     & \ddots & c      & 0 \\
    0      & \cdots & \cdots & 0      & I
\end{array} \right]
\begin{array}{l}
   \vspace{-.3cm}\\ \textrm{row $k$} \\ \vspace{.45cm} \\ \textrm{row $l$}
\end{array} \]
where $I$ and $0$ denote identity and zero matrices of appropriate size respectively,
and where
\[
c = \frac{\mathbf{y}_k}{\sqrt{\mathbf{y}_k^2 + \mathbf{y}_l^2}} \textrm{~~and~~}
s = \frac{\mathbf{y}_l}{\sqrt{\mathbf{y}_k^2 + \mathbf{y}_l^2}} .
\]
The composition of this matrix is such that if $\mathbf{\tilde{y}}=G_\mathbf{y}\left(k,l\right)\mathbf{y}$,
then $\mathbf{\tilde{y}}_i=\mathbf{y}_i$ for all $i\not\in \{k,l\}$, and that $\mathbf{\tilde{y}}_l=0$.

We define the following shorthand notation for the Givens rotations we are going to use:
\[ G_i = G_{\boldsymbol{\tau}_j^i} \left(i,i+1\right),~~ i = 1,\ldots,j, \]
where the transformation vector $\boldsymbol{\tau}_j^i$ is given by
$\boldsymbol{\tau}_j^i = G_{i-1} \cdots G_1 \mathbf{t}_j^i$.
Here $\mathbf{t}_j^i$ denotes column $i$ of the extended Ritz matrix $\tilde{T}_j$,
and thus $\boldsymbol{\tau}_j^i$ is this column of the
extended Ritz matrix after application of the rotations $1,\ldots,i-1$.
Note that the rotation $G_i$ is the same for all $j$, except for its dimensions.
With $G_i$ we mean the rotation matrix of appropriate size.

Using the above rotations, we define the transformed matrix
\begin{equation}\label{eq:ritz_rotated}
\tilde{U}_j = G_j \cdots G_1 \tilde{T}_j .
\end{equation}
Due to the structure of $\mathbf{t}_j^i$, and defining $G_{-1}=G_0=I$,
the columns of $\tilde{U}_j$ are given by
\begin{equation}\label{eq:u_transformation}
\mathbf{\tilde{u}}_j^i = G_i \cdots G_1 \mathbf{t}_j^i
= G_i G_{i-1} G_{i-2} \mathbf{t}_j^i = G_i G_{i-1} G_{i-2}
\left[\begin{array}{c} \mathbf{t}_i^i \\ \mathbf{0}_{j-i} \end{array}\right] =
\left[\begin{array}{c} \mathbf{\tilde{u}}_i^i \\ \mathbf{0}_{j-i} \end{array}\right].
\end{equation}
From an algorithmic point of view, this means that in iteration $j>1$
we can construct $\mathbf{\tilde{u}}_j^i$ for $i<j$ directly from $\mathbf{\tilde{u}}_{j-1}^i$,
without having to apply Givens rotations.
The only vector that has to be calculated using these rotations is $\mathbf{\tilde{u}}_j^j$,
and this can be done using only the rotations $G_j$, $G_{j-1}$ and $G_{j-2}$.

Note that $\mathbf{\tilde{u}}_i^i$ has dimension $i+1$,
and that $\mathbf{\tilde{u}}_i^i\left(i+1\right)=0$ due to the rotations.
Thus we can implicitely define the matrix $U_j$, by
\[ \tilde{U}_j = \left[ \begin{array}{c} U_j \\ 0 \cdots 0 \end{array} \right]. \]

\begin{theorem}
The matrix $U_j$ is a $j \times j$ matrix with the following sparsity structure:
\[ U_j = \left[ \begin{array}{ccccc}
    * & 0 & * &   &  \\
      & * & 0 & * &  \\
      &   & * & 0 & * \\
      &   &   & * & 0 \\
      &   &   &   & *
\end{array} \right]. \]
\end{theorem}

\begin{proof}

Define
\[
\renewcommand\arraystretch{1.4}
\begin{array}{lcl}
Z_1 &=& \alpha^2 \,, \\
Z_2 &=& Z_1 + \beta_2^2 \,, \\
Z_3 &=& Z_2 + \beta_3^2 \,, \\

Z_i &=& \frac{Z_1 Z_3 \cdots Z_{i-1}}{Z_2 Z_4 \cdots Z_{i-2}} + \beta_i^2 \,,~~ i > 3 \,,~ i \textrm{~even} \,, \\
Z_i &=& \frac{Z_2 Z_4 \cdots Z_{i-1}}{Z_3 Z_5 \cdots Z_{i-2}} + \beta_i^2 \,,~~ i > 3 \,,~ i \textrm{~odd} \,. \\
\end{array}
\renewcommand\arraystretch{1}
\]

Let, as in Section~\ref{sec:idemasolver_solving},
\[ G_j^i = G_{\boldsymbol{\tau}_j^i}\left(i, i+1\right) ~,~~ j \geq i \]
denote Givens rotation $i$ at iteration $j$, and let
\[
\renewcommand\arraystretch{1.4}
\begin{array}{rcl}
c_i &=& \frac{\boldsymbol{\tau}_j^i\left(i\right)}{\sqrt{\boldsymbol{\tau}_j^i\left(i\right)+\boldsymbol{\tau}_j^i\left(i+1\right)}} \\
s_i &=& \frac{\boldsymbol{\tau}_j^i\left(i+1\right)}{\sqrt{\boldsymbol{\tau}_j^i\left(i\right)+\boldsymbol{\tau}_j^i\left(i+1\right)}} \\
\end{array}
\renewcommand\arraystretch{1}
\]
denote the coefficients of $G_j^i$.
Note that for $j \geq i$ the values of $c_i$ and $s_i$ are indeed independent of $j$
due to the special structure of $\boldsymbol{\tau}_j^i$.

For $j=1$ we have $\boldsymbol{\tau}_1^1 = \textrm{t}_1^1$ and find
\[ c_1 = \frac{\sqrt{Z_1}}{\sqrt{Z_2}} ~,~~ s_1 = \frac{-\beta_2}{\sqrt{Z_2}} \,, \]
\[
\textbf{u}_1^1 = G_1^1\boldsymbol{\tau}_1^1 = G_1^1\textrm{t}_1^1 =
\left[\begin{array}{c}
\alpha c_1 - \beta_2 s_1 \\ -\alpha s_1 - \beta_2 c_1

\end{array}\right] =
\left[\begin{array}{c}
\sqrt{Z_2} \\ 0
\end{array}\right] .
\]

For $j=2$ we have
\[
\boldsymbol{\tau}_2^2 = G_2^1\textrm{t}_2^2 = 
\left[\begin{array}{c}
\beta_2 c_1 + \alpha s_1 \\ -\beta_2 s_1 + \alpha c_1 \\ -\beta_3
\end{array}\right] =
\left[\begin{array}{c}
0 \\ \sqrt{Z_2} \\ -\beta_3
\end{array}\right] .
\]
Thus we find
\[ c_2 = \frac{\sqrt{Z_2}}{\sqrt{Z_3}} ~,~~ s_2 = \frac{-\beta_3}{\sqrt{Z_3}} \,, \]
\[
\textbf{u}_2^2 = G_2^2\boldsymbol{\tau}_2^2 =
\left[\begin{array}{c}
0 \\ \sqrt{Z_2}c_2 - \beta_3 s_2 \\ -\sqrt{Z_2}s_2 - \beta_3 c_2
\end{array}\right] =
\left[\begin{array}{c}
0 \\ \sqrt{Z_3} \\ 0
\end{array}\right] .
\]

Now assume that
\begin{equation}\label{eq:givens_coefficients}
\renewcommand\arraystretch{1.4}
\begin{array}{lcl}
c_i &=& \sqrt{\frac{Z_2 Z_4 \cdots Z_i}{Z_3 Z_5 \cdots Z_{i+1}}} \,,~ i \textrm{~even} \,, \\
c_i &=& \sqrt{\frac{Z_1 Z_3 \cdots Z_i}{Z_2 Z_4 \cdots Z_{i+1}}} \,,~ i \textrm{~odd} \,, \\
s_i &=& \frac{-\beta_{i+1}}{\sqrt{Z_{i+1}}} \,.
\end{array}
\renewcommand\arraystretch{1}
\end{equation}
Obviously this is true for $i=1,2$.
We will show by induction that it holds for all $i>0$.

For $j>2$ we have
\[ \boldsymbol{\tau}_j^j = G_j^{j-1}G_j^{j-2}\textrm{t}_j^j = 
G_j^{j-1}\left[\begin{array}{c}
\textbf{0}_{j-3} \\ \beta_j s_{j-2} \\ \beta_j c_{j-2} \\ \alpha \\ -\beta_{j+1}

\end{array}\right] =
\left[\begin{array}{c}
\textbf{0}_{j-3} \\ \beta_j s_{j-2} \\ \beta_j c_{j-2}c_{j-1} + \alpha s_{j-1} \\
-\beta_j c_{j-2}s_{j-1} + \alpha c_{j-1} \\ -\beta_{j+1}
\end{array}\right] , \]
where $\textbf{0}_i$ denotes the null vector with dimension $i$.

Suppose that $j$ is even. Using~(\ref{eq:givens_coefficients}) we then find \\
\[
\renewcommand\arraystretch{1.4}
\begin{array}{lcl}

\beta_j c_{j-2}c_{j-1} + \alpha s_{j-1} &=& \\
\beta_j \sqrt{\frac{Z_2 Z_4 \cdots Z_{j-2}}{Z_3 Z_5 \cdots Z_{j-1}}}
\sqrt{\frac{Z_1 Z_3 \cdots Z_{j-1}}{Z_2 Z_4 \cdots Z_j}}
+ \alpha \frac{-\beta_j}{\sqrt{Z_j}} &=& \\
\beta_j \sqrt{\frac{Z_1}{Z_j}} + \alpha \frac{-\beta_j}{\sqrt{Z_j}} &=& 0 \,.

\renewcommand\arraystretch{1}
\end{array}
\]
Further we can write
\[
\renewcommand\arraystretch{1.8}
\begin{array}{lcl}
-\beta_j c_{j-2}s_{j-1} + \alpha c_{j-1} &=& \\
-\beta_j \sqrt{\frac{Z_2 Z_4 \cdots Z_{j-2}}{Z_3 Z_5 \cdots Z_{j-1}}} \frac{-\beta_j}{\sqrt{Z_j}}
+ \alpha \sqrt{\frac{Z_1 Z_3 \cdots Z_{j-1}}{Z_2 Z_4 \cdots Z_j}} &=& \\
\alpha \sqrt{\frac{Z_1 Z_3 \cdots Z_{j-1}}{Z_2 Z_4 \cdots Z_j}} +
\frac{\beta_j^2}{\sqrt{Z_j}} \sqrt{\frac{Z_2 Z_4 \cdots Z_{j-2}}{Z_3 Z_5 \cdots Z_{j-1}}} &=& \\
\frac{Z_1}{\sqrt{Z_j}} \sqrt{\frac{Z_3 \cdots Z_{j-1}}{Z_2 Z_4 \cdots Z_j-2}} +
\frac{\beta_j^2}{\sqrt{Z_j}} \sqrt{\frac{Z_2 Z_4 \cdots Z_{j-2}}{Z_3 Z_5 \cdots Z_{j-1}}} &=& \\
\sqrt{\frac{Z_2 Z_4 \cdots Z_{j-2}}{Z_3 Z_5 \cdots Z_{j-1}}}
\left( \frac{\frac{Z_1Z_3 \cdots Z_{j-1}}{Z_2 Z_4 \cdots Z_j-2}}{\sqrt{Z_j}} +
\frac{\beta_j^2}{\sqrt{Z_j}} \right) &=& \\
\sqrt{\frac{Z_2 Z_4 \cdots Z_{j-2}}{Z_3 Z_5 \cdots Z_{j-1}}}\sqrt{Z_j} &=&
\sqrt{\frac{Z_2 Z_4 \cdots Z_j}{Z_3 Z_5 \cdots Z_{j-1}}} \,.
\renewcommand\arraystretch{1}
\end{array}
\]

Therefore, assuming~(\ref{eq:givens_coefficients}) for $i=1\ldots{j-1}$, for even $j>2$ we have
\[ \boldsymbol{\tau}_j^j = 
\left[\begin{array}{c}
\textbf{0}_{j-3} \\ \frac{-\beta_{j-1}\beta_j}{Z_{j-1}} \\ 0 \\
\sqrt{\frac{Z_2 Z_4 \cdots Z_j}{Z_3 Z_5 \cdots Z_{j-1}}} \\ -\beta_{j+1}
\end{array}\right] . \]
From this it easily follows that indeed for $j$, again assumption~(\ref{eq:givens_coefficients}) holds, and that
\[
\textbf{u}_j^j = G_j^j\boldsymbol{\tau}_j^j =
\left[\begin{array}{c}
\textbf{0}_{j-3} \\ \frac{-\beta_{j-1}\beta_j}{Z_{j-1}} \\ 0 \\
\sqrt{\frac{Z_2 Z_4 \cdots Z_j}{Z_3 Z_5 \cdots Z_{j-1}}}c_j - \beta_{j+1}s_j \\ 0
\end{array}\right] =
\left[\begin{array}{c}
\textbf{0}_{j-3} \\ \frac{-\beta_{j-1}\beta_j}{Z_{j-1}} \\ 0 \\ \sqrt{Z_{j+1}} \\ 0
\end{array}\right] .
\]

In the same way we can prove that for odd $j>2$ we have
\[ \boldsymbol{\tau}_j^j = 
\left[\begin{array}{c}
\textbf{0}_{j-3} \\ \frac{-\beta_{j-1}\beta_j}{Z_{j-1}} \\ 0 \\
\sqrt{\frac{Z_1 Z_3 \cdots Z_j}{Z_2 Z_4 \cdots Z_{j-1}}} \\ -\beta_{j+1}
\end{array}\right] , \]
that thus assumption~(\ref{eq:givens_coefficients}) holds for all $j>0$,
and that for odd $j>2$ again
\[
\textbf{u}_j^j =
\left[\begin{array}{c}
\textbf{0}_{j-3} \\ \frac{-\beta_{j-1}\beta_j}{Z_{j-1}} \\ 0 \\ \sqrt{Z_{j+1}} \\ 0
\end{array}\right] .
\]

\end{proof}

Note that the zeros on the first superdiagonal of $U_j$ are non-trivial,
and will result in a 2-term recursion for the calculation of $\mathbf{s}_j$,
instead of the expected 3-term recursion.

Now let us define the $\left(j+1\right)$-dimensional rotated vector
\begin{equation}\label{eq:rhs_rotated}
\mathbf{\tilde{v}}_j = - G_j \cdots G_1 \left|\left|\mathbf{r}_0\right|\right|_2\mathbf{e}_1.
\end{equation}
Note that, defining $\mathbf{\tilde{v}}_0 = -\left|\left|\mathbf{r}_0\right|\right|_2$, we can write
$\mathbf{\tilde{v}}_j = G_j \left[ \mathbf{\tilde{v}}_{j-1} ~~ 0 \right]^T,~j>0$.
Further, writing
$\mathbf{\tilde{v}}_j = \left[ \mathbf{v}_j ~~ \varepsilon_j \right]^T$,
we have
\begin{equation}\label{eq:v_construction}
\mathbf{\tilde{v}}_j
= G_j \left[ \begin{array}{l} \mathbf{v}_{j-1} \\ \varepsilon_{j-1} \\ 0 \end{array} \right]
= \left[ \begin{array}{l} \mathbf{v}_{j-1} \\ \mu_j \\ \varepsilon_j \end{array} \right]
~~\textrm{and}~~
\mathbf{v}_j = \left[ \begin{array}{l} \mathbf{v}_{j-1} \\ \mu_j \end{array} \right].
\end{equation}

Since a Givens rotation is an orthogonal transformation,
using equation~(\ref{eq:res_in_xi}) and
expressions~(\ref{eq:ritz_rotated}) and~(\ref{eq:rhs_rotated}), we can write
\[ \left|\left|\mathbf{r}_j\right|\right|_2
= \left|\left| \, (\tilde{T}_j\boldsymbol{\xi}_j
+ \left|\left|\mathbf{r}_0\right|\right|_2\mathbf{e}_1) \, \right|\right|_2
= \left|\left| \, \tilde{U}_j\boldsymbol{\xi}_j
- \mathbf{\tilde{v}}_j \, \right|\right|_2 . \]
Thus the solution $\boldsymbol{\hat\xi}_j$ of equation~(\ref{eq:minres_xi}),
is equal to the least-squares solution of the system
\[ \tilde{U}_j \boldsymbol{\xi}_j = \mathbf{\tilde{v}}_j \Leftrightarrow
\left[ \begin{array}{c} U_j \\ 0 \cdots 0 \end{array} \right]\boldsymbol{\xi}_j
= \left[ \begin{array}{l} \mathbf{v}_j \\ \varepsilon_j \end{array} \right] . \]

From this result it is trivial that $\boldsymbol{\hat\xi}_j$ is the solution of the system
\begin{equation}\label{eq:minres_system2}
U_j\boldsymbol{\xi}_j = \mathbf{v}_j,
\end{equation}
and that the residual error is given by
\begin{equation}\label{eq:minres_res}
\left|\left|\mathbf{r}_j\right|\right|_2 = \left|\varepsilon_j\right|.
\end{equation}

To determine the minimal residual approximation $\mathbf{x}_j$
we now need to calculate $\mathbf{s}_j=Q_j\boldsymbol{\hat\xi}_j$.
If we calculate $\boldsymbol{\hat\xi}_j$ as the solution of system~(\ref{eq:minres_system2}),
and then multiply by $Q_j$ directly, we would need to store the entire matrix $Q_j$ in memory.
Thus the algorithm would use long recurrences.
To overcome this problem we can use the technique that is also applied in
the MINRES algorithm~\cite{paige_saunders:minres}, as detailed below.

Define the matrix $W_j = Q_j U_j^{-1}$, then
\begin{equation}\label{eq:w_implicit}
W_j U_j = Q_j,
\end{equation}
and
\begin{equation}\label{eq:s_decomp_w}
\mathbf{s}_j = Q_j U_j^{-1} \mathbf{v}_j = W_j \mathbf{v}_j.
\end{equation}
Further, introduce the notations
$W_j =   \left[ \mathbf{w}_j^1 \cdots \mathbf{w}_j^j \right]$ and
$W_j^i = \left[ \mathbf{w}_j^1 \cdots \mathbf{w}_j^i \right]$.

For $j=1$ equation~(\ref{eq:w_implicit}) has the unique solution
$\mathbf{w}_1^1 u_{1,1} = \mathbf{q}_1 \Rightarrow \mathbf{w}_1^1 = \frac{1}{u_{1,1}} \mathbf{q}_1$.
Now assume that $j=i$ with $i>1$,
and that we have a unique solution of equation~(\ref{eq:w_implicit}) for $j=i-1$, then
\[ \left[ \, W_i^{i-1} ~~ \mathbf{w}_i^i \, \right]
\left[ \, \begin{array}{c} U_{i-1} \\ 0 \cdots 0 \end{array} ~ \mathbf{u}_i^i \, \right]
= \left[ \, Q_{i-1} ~~ \mathbf{q}_i \, \right], \]
which can be split in the equations
\begin{eqnarray}
W_i^{i-1} U_{i-1} &=& Q_{i-1} \label{eq:w_implicit_split1} \\
W_i \mathbf{u}_i^i &=& \mathbf{q}_i \label{eq:w_implicit_split2} \,.
\end{eqnarray}

Equation~(\ref{eq:w_implicit_split1}) has the unique solution
$W_i^{i-1} = W_{i-1}$, and due to the special structure of $\mathbf{u}_i^i$
equation~(\ref{eq:w_implicit_split2}) is easily solved.
Thus, by induction, we find that we can unambiguously define $\mathbf{w}_i = \mathbf{w}_j^i$,
and that equation~(\ref{eq:w_implicit}) is uniquely solved, in a 2-term recurrence,
by the matrix $W_j$ with columns
\begin{equation}\label{eq:w_solution}
\mathbf{w}_i = \left\{
\begin{array}{l}
\displaystyle \frac{1}{u_{i,i}} \mathbf{q}_i,~i \in \left\{1,2\right\} \\
\displaystyle \frac{1}{u_{i,i}} \left( \mathbf{q}_i-u_{i-2,i}\mathbf{w}_{i-2} \right),~i>2
\end{array}
\right.
\end{equation}

The final step is to find the approximating solution $\mathbf{x}_j$.
Using equations~(\ref{eq:v_construction}) and~(\ref{eq:s_decomp_w}) we can write
\[ \mathbf{s}_j = W_j \mathbf{v}_j
= W_{j-1} \mathbf{v}_{j-1} + \mu_j \mathbf{w}_j
= \mathbf{s}_{j-1} + \mu_j \mathbf{w}_j . \]
Thus the approximation of the solution in iteration $j$ is given by
\begin{equation}\label{x_solution}
\mathbf{x}_j = \mathbf{x}_0 + \mathbf{s}_j = \mathbf{x}_{j-1} + \mu_j \mathbf{w}_j.
\end{equation}

Combining all the above results, we now present the {\solvername} solver Algorithm~\ref{algo:sss_solver}.
To make the algorithm easier to read, we have used the simplified notations
$\mathbf{u}_j$ for $\mathbf{\tilde{u}}_j$ and $\mathbf{v}_j$ for $\mathbf{\tilde{v}}_j$.
Further note that $\mathbf{u}_j = \left[ \, \mathbf{0}_{j-2} ~~ \beta_j ~~ \alpha ~~ -\beta_{j+1} \, \right]^T$
should be read such that $\mathbf{u}_1 = \left[ \, \alpha ~~ -\beta_{j+1} \, \right]^T$,
and $\mathbf{u}_2 = \left[ \, \beta_{j} ~~ \alpha ~~ -\beta_{j+1} \, \right]^T$,
and that we define $\mathbf{u}_j(k)=0$ for $k \leq 0$.

\begin{figure}[ht]
\hrulefill
\vspace{-.2cm}
\begin{algorithm}[{\solvername}]\label{algo:sss_solver} ~\\
\algospace
\begin{list}{}{}
\item Choose $\mathbf{x}_0$ and set the residual error tolerance $\tau$ \\
Let $\mathbf{r}_0=\mathbf{b}-A\mathbf{x}_0$, $S = A -\alpha I$, $j=0$ \\
Let $\mathbf{q}_0=0$, $\mathbf{p}_1=\mathbf{r}_0$,
$\varepsilon=\beta_1=\left|\left|\mathbf{p}_1\right|\right|_2$,
$\mathbf{v}_0=\left[-\beta_1\right]$, $G_{-1}=G_{0}=I, \mathbf{w}_{-1}=\mathbf{w}_{-0}=\mathbf{0}$

\item While $\varepsilon > \tau$ do

    \begin{list}{}{}
    \item $j = j+1$
    \item $\mathbf{q}_j = -\mathbf{p}_j / \beta_j$
    \item $\mathbf{p}_{j+1} = S \mathbf{q}_j - \beta_j \mathbf{q}_{j-1}$
    \item $\beta_{j+1} = \left|\left|\mathbf{p}_{j+1}\right|\right|_2$
    \end{list}

    \begin{list}{}{}
    \item $\mathbf{u}_j = \left[ \, \mathbf{0}_{j-2} ~~ \beta_j ~~ \alpha ~~ -\beta_{j+1} \, \right]^T$
    \item $G_j = GivensRotation(\mathbf{u}_j, j, j+1)$
    \item $\mathbf{u}_j = G_jG_{j-1}G_{j-2}\mathbf{u}_j$
    \item $\mathbf{v}_j = \left[ \, \mathbf{v}_{j-1} ~~ 0 \, \right]^T$
    \item $\mathbf{v}_j = G_j\mathbf{v}_j$
    \end{list}

    \begin{list}{}{}
    \item $\mathbf{w}_j = (\mathbf{q}_j - \mathbf{u}_j(j-2)\mathbf{w}_{j-2}) / \mathbf{u}_j(j)$
    \item $\mathbf{x}_j = \mathbf{x}_{j-1} + \mathbf{v}_j(j) \mathbf{w}_j$
    \item $\varepsilon = \mathbf{v}_j(j+1)$
    \end{list}

\item End while
\end{list}
\end{algorithm}
\vspace{-.6cm}\hrulefill\vspace{.2cm}
\end{figure}

\section{Theoretical comparison}\label{sec:solver_comparison}

The GMRES algorithm is based on the Arnoldi method, combined with Givens rotations.
As noted in Section~\ref{sec:idemasolver_lanczos_nonsymmetric},
for SSS matrices the Arnoldi method reduces to the shifted skew-symmetric Lanczos algorithm.
Thus it is clear that the first couple of algorithmic steps of {\solvername},
as described in the previous section, are equal to those of the
GMRES algorithm with the orthogonalization truncated after one step,
and using only the last three Givens rotations.

In the final step, GMRES calculates
the matrix vector product $\mathbf{s}_j=Q_j\boldsymbol{\hat\xi}_j$ directly.
As noted this constitutes a long recurrence algorithm, even for SSS systems.
The {\solvername} solver, instead uses a technique also applied in MINRES
to calculate the same update to the approximate solution with short recurrences.
In this light, it is clear that {\solvername} can also be seen
as a shifted skew-symmetric version of MINRES.

The GCR algorithm, and thus for SSS systems Orthomin(1), also minimizes the residual
within the Krylov subspace $\mathcal{K}_j(A,\mathbf{r}_0)$ in each step.
Assuming that $A$ is non-singular and that GCR does not break down, this implies that in exact arithmetic
it generates the same approximations to the solution as GMRES and {\solvername}.
Thus we have
\[ \mathbf{x}_j^\textrm{\solvername} = \mathbf{x}_j^\textrm{GMRES} = \mathbf{x}_j^\textrm{GCR}. \]

As {\solvername} uses short recurrences, it is more efficient in finding these iterates than GMRES.
GCR is also very efficient, but it breaks down for skew-symmetric systems, i.e., for $\alpha=0$.
Also, we expect that for small $\alpha$ GCR will have problems,
generating a direction $\beta_j\mathbf{v}_j$ with very small norm $\beta_j$.
As a result, rounding errors will blow up when dividing the generated direction by its norm.

The relation between the residual of a minimal residual method MR,
and that of a Galerkin method G, is known from~(2.29) of~\cite{Gut01R}.
If the minimal residual method does not stagnate, i.e., if
\[ c = \frac{\left|\left|\mathbf{r}_j^\textrm{MR}\right|\right|_2}
{\left|\left|\mathbf{r}_{j-1}^\textrm{MR}\right|\right|_2} < 1, \]
then the norm of the residual satisfies the identity
\begin{equation}\label{eq:relation_mr_galerkin}
\left|\left|\mathbf{r}_j^\textrm{MR}\right|\right|_2
= \sqrt{1-c^2} \left|\left|\mathbf{r}_j^\textrm{G}\right|\right|_2 .
\end{equation}
It follows directly that, if the {\solvername} method does not stagnate,
the calculated residuals are always smaller than those of a Galerkin method like CGW, thus
\[ \left|\left|\mathbf{r}_j^\textrm{\solvername}\right|\right|_2
< \left|\left|\mathbf{r}_j^\textrm{CGW}\right|\right|_2 . \]
Furthermore, relation~(\ref{eq:relation_mr_galerkin}) can be used to understand
the so-called peak-plateau connection \cite{Cul95,vdV93V,Wal95}.
The peak-plateau connection is the phenomenon that a peak in the residual norm history of a Galerkin method
is accompanied by a plateau, i.e., the norm nearly stagnating, in the residual norm history of a minimal residual method.

Bi-CGSTAB does not satisfy the optimality property,
and will thus in general converge slower than the minimal residual methods described.
On the other hand, general methods like Bi-CGSTAB, but also GMRES and CGNR,
offer more preconditioning options than the algorithms specifically tailored for SSS systems.

CGNR can be expected to converge very fast as long as the problem is well-conditioned.
The work by Greif and Varah~{\cite{greif_varah:iter_sol_skew_symm}}
contains some interesting insights in the use of normal equations to solve SSS systems with $\alpha=0$,
as well as work on preconditioners for solvers that are designed to deal with skew-symmetric systems.
A possible alternative for, or addition to, preconditioning,
is the regularization technique for CGNR, described in~\cite{chen_shen:regularized_cgnr}.
This technique can readily be combined with {\solvername} or any other method, instead of CGNR.

Table~\ref{tab:solver_properties} gives an overview of some important properties of the computational load,
as well as some general properties of {\solvername} and other algorithms treated in the previous chapters.
The amount of computational work is measured for a single iteration $j$.
Note that these numbers can vary with the exact implementation of the algorithm.\footnote{See
\url{http://ta.twi.tudelft.nl/nw/users/idema/mrs3/} for MATLAB implementations of these methods.}
Further note that with optimality we mean that the method satisfies
the optimality property within the Krylov subspace $\mathcal{K}_j(A,\mathbf{r}_0)$.

\begin{table}[!htb]
\begin{center}
\begin{tabular}{|c|c|c|c|c|c|c|}
\hline
& matvec   & vector   & inner    & vector & & \\
& products & updates  & products & memory & $\alpha$ & optimality  \\
\hline
{\solvername} & 1 & 3 & 1 & 5 & all & yes \\
CGW           & 1 & 3 & 2 & 4 & $\neq 0$ & no \\
Trunc-GCR     & 1 & 4 & 4 & 5 & $\neq 0$ & yes \\
GMRES         & 1 & $\frac{j+1}{2}$ & $\frac{j+1}{2}$ & $j+3$ & all & yes \\
Bi-CGSTAB     & 2 & 3 & 3 & 7 & all & no \\
CGNR          & 2 & 3 & 3 & 5 & all & no \\
\hline
\end{tabular}
\caption{Important properties of {\solvername} and other solvers}
\label{tab:solver_properties}
\end{center}
\end{table}

\section{Numerical results}\label{sec:numerical_experiments}

In this section we compare the {\solvername} method
with the CGW, GCR, GMRES, Bi-CGSTAB and CGNR methods numerically,
by solving some SSS systems $A\mathbf{x}=\mathbf{b}$ and analysing the residual norm history.
Also, we will verify numerically the theoretical results from Section~\ref{sec:solver_comparison}.

As mentioned in the introduction of this paper,
SSS systems frequently occur in the solution of advection-diffusion problems.
For our numerical experiments, we use matrices that correspond to a
finite difference discretisation of the following partial differential equation:
\[ \frac{\partial{u}}{\partial{x}} + \gamma\frac{\partial{u}}{\partial{y}} = f \]
with $(x,y) \in [0,1] \times [0,1]$ and appropriate boundary conditions.
The number of gridpoints in $x$ and $y$ direction are denoted by $n_1$ , $n_2$ respectively.

The resulting matrices are of the form $A = \alpha I + S$,
where $A\in\mathbb{R}^{n \times n}$, $n = n_1n_2$, $h_1 = \frac{1}{n_1}$, $h_2 = \frac{1}{n_2}$,
and with the matrix $S$ a skew-symmetric block tridiagonal matrix,
of which the $n_1 \times n_1$ nonzero blocks are given by
\begin{eqnarray*}
S_{i,i} &=& \frac{1}{2h_1}\textrm{tridiag}\left(-1,0,1\right), \textrm{ for } i = 1, \hdots, n_2 , \\
S_{i,i+1} &=& -S_{i+1,i} ~=~ \frac{1}{2h_2}\textrm{diag}\left(\gamma\right), \textrm{ for } i = 1, \hdots, n_2 -1 .
\end{eqnarray*}
For this discussion, we use $n_1=n_2=20$ and vary $\alpha$ and $\gamma$.
For practical validation, experiments with much larger dimensions were done.
These led to the same conclusions as presented below.

The starting approximation is chosen to be $\mathbf{x}_0=\mathbf{0}$,
and the right-hand side vector $\mathbf{b}$ consists of random numbers, and is normalized such that
$\left|\left|\mathbf{r}_0\right|\right|_2 = \left|\left|\mathbf{b}\right|\right|_2 = 1$.
Our special interest goes to ill-conditioned systems, i.e.,
systems with a coefficient matrix $A$ with very large condition number $\kappa$,
and systems with small $\alpha$.

\subsection{Numerical comparison of {\solvername}, GCR and CGNR}

{\solvername}, GCR and CGNR are all short recurrence algorithms.
In our tests all three methods showed comparable convergence for well-conditioned problems with large $\alpha$.
However, Figure~\ref{fig:num_comp1a} shows that for an ill-conditioned system with small $\alpha$,
both GCR and CGNR cannot keep up with the performance of {\solvername}.
In Figure~\ref{fig:num_comp1b} the performance of CGNR is equal to that of {\solvername}.
This demonstrates that, if the system is well-conditioned,
CGNR can perform as well as {\solvername} for small $\alpha$,
while GCR does not converge to an accurate solution.

\begin{figure}[!htb]
\centering
\subfigure[$\alpha=10^{-6}$, $\gamma=1$, $\kappa=4\cdot 10^7$]
{
\includegraphics[height=5.7cm]{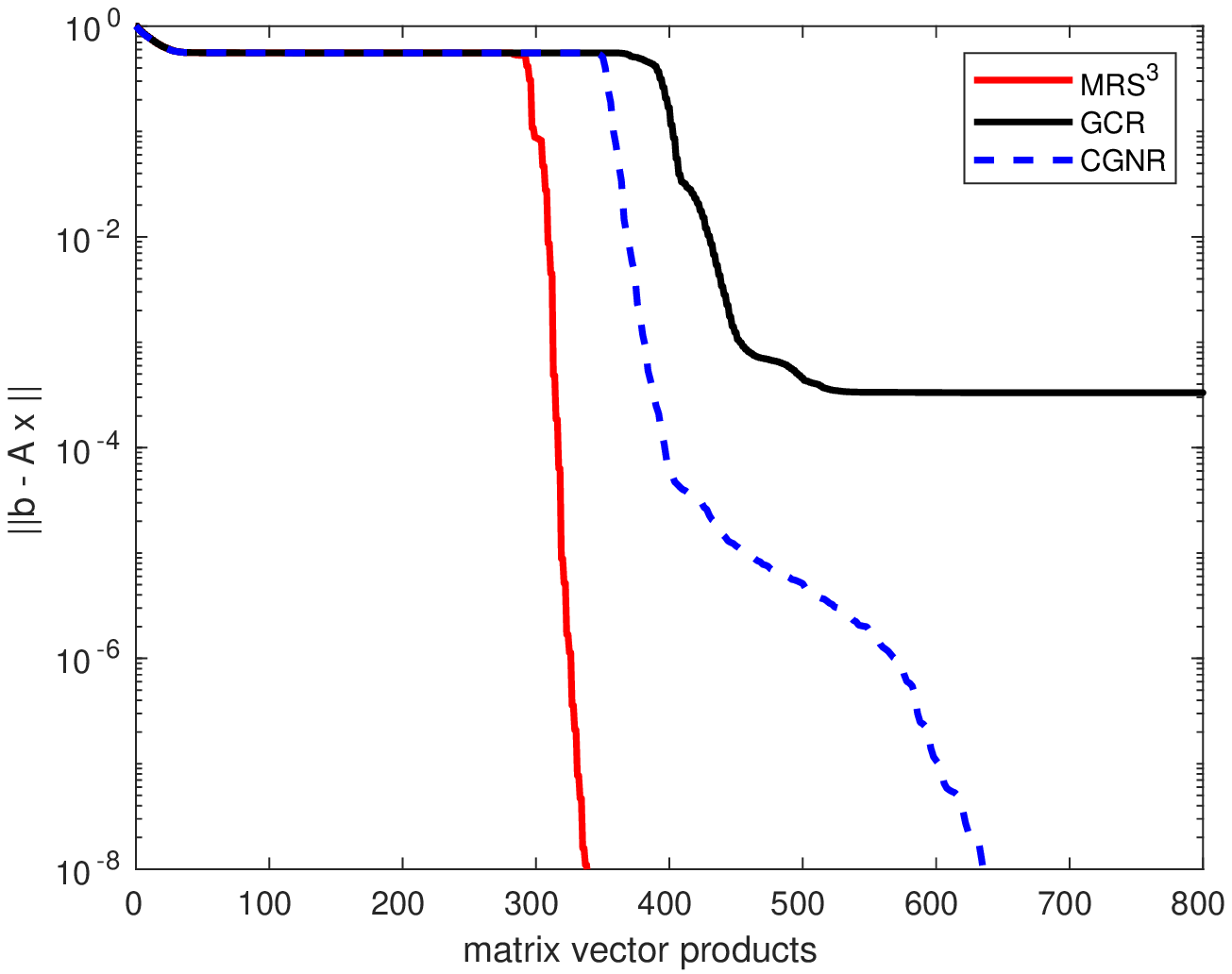}\label{fig:num_comp1a}
}
\subfigure[$\alpha=10^{-5}$, $\gamma=100$, $\kappa=15$]
{
\includegraphics[height=5.7cm]{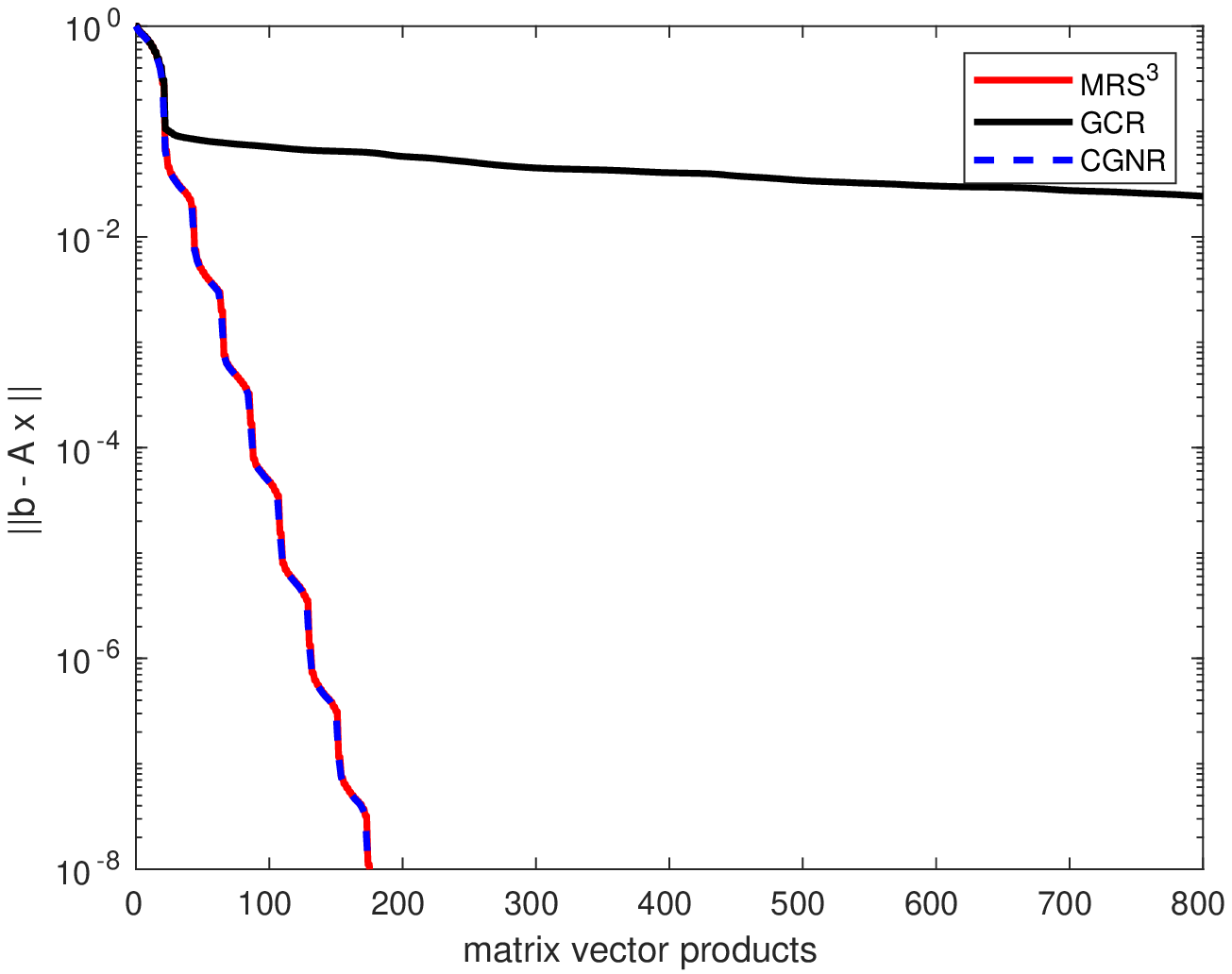}\label{fig:num_comp1b}
}
\caption{Convergence of {\solvername}, GCR and CGNR}
\end{figure}

\subsection{Numerical comparison of {\solvername}, GMRES(3) and Bi-CGSTAB}

We compare {\solvername} with the general Krylov methods GMRES and Bi-CGSTAB.
In order to make the memory and work requirements comparable we use GMRES(3),
which means that GMRES is restarted every 3 iterations.
In our experiments we have also checked that full GMRES
indeed leads to the same numerical results as {\solvername}.

As expected, our experiments show that GMRES(3) and Bi-CGSTAB
do not perform very well for SSS systems compared to {\solvername},
see Figures~\ref{fig:num_comp2a} and~\ref{fig:num_comp2b}.
With respect to Bi-CGSTAB, we should note that for these problems
BiCGStab2~\cite{Gut93} and Bi-CGSTAB($\ell$)~\cite{Sle94VF} may be better alternatives.

\begin{figure}[!htb]
\centering
\subfigure[$\alpha=10$, $\gamma=1$, $\kappa=4$]
{
\includegraphics[height=5.7cm]{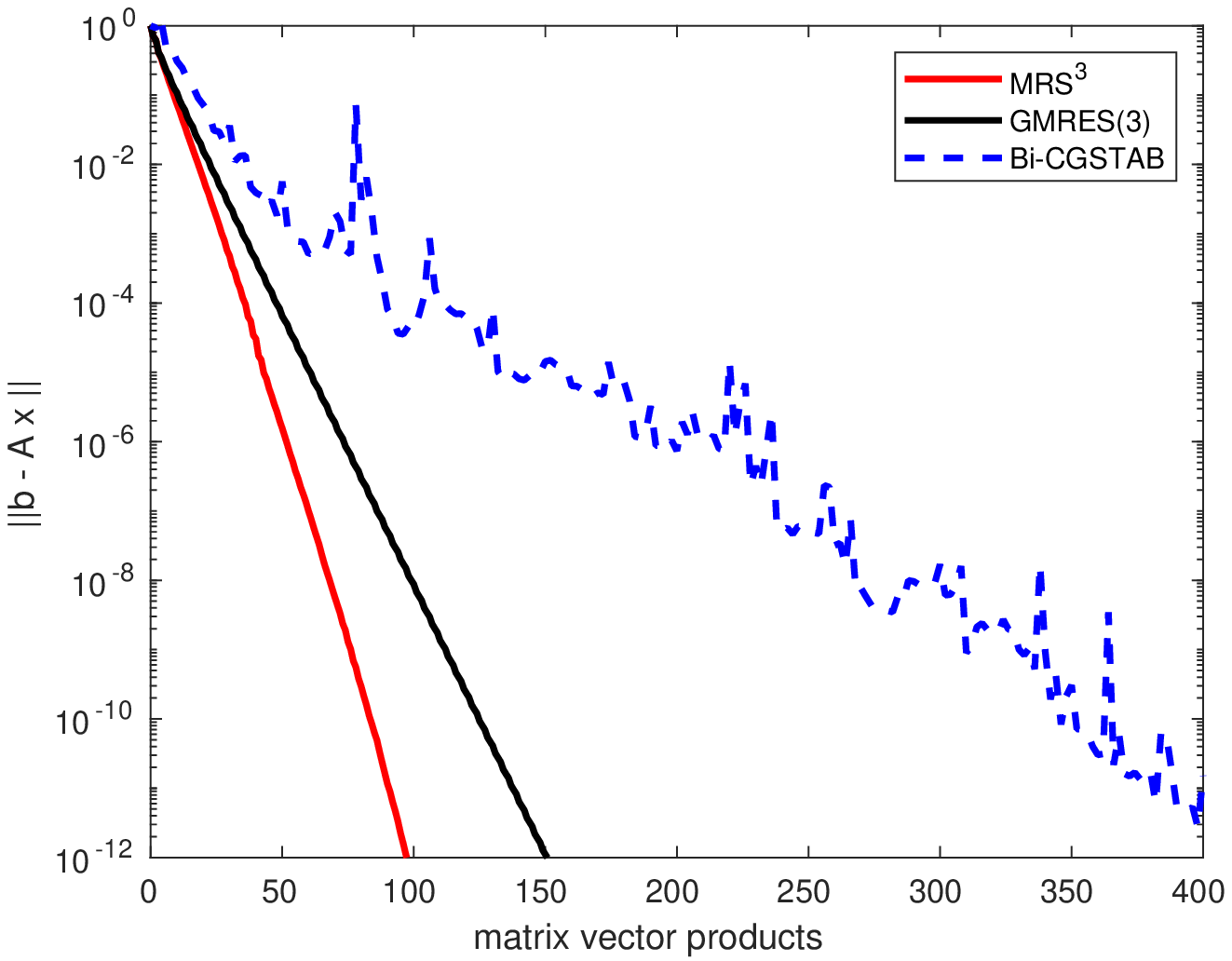}\label{fig:num_comp2a}
}
\subfigure[$\alpha=10^{-3}$, $\gamma=1$, $\kappa=4\cdot 10^4$]
{
\includegraphics[height=5.7cm]{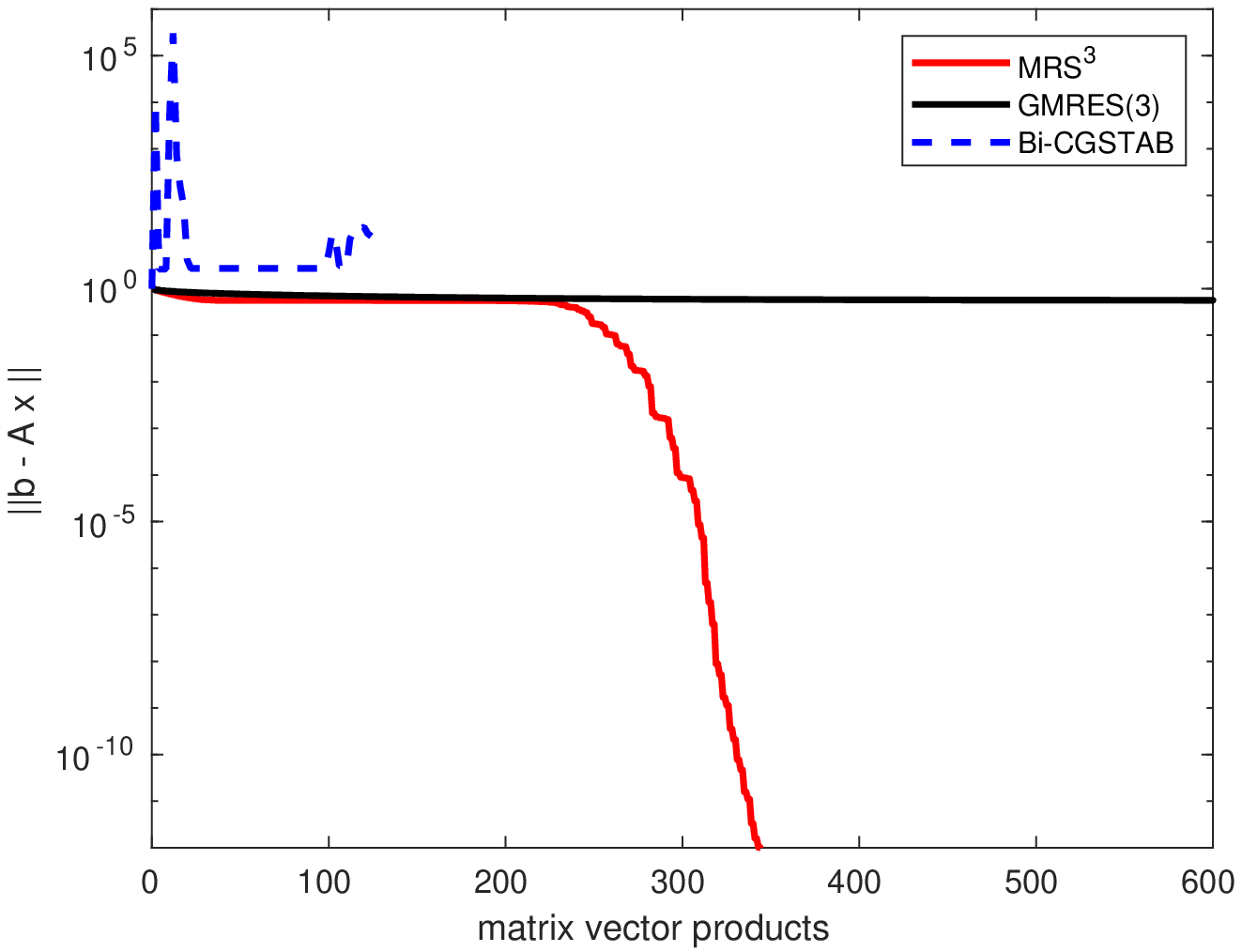}\label{fig:num_comp2b}
}
\caption{Convergence of {\solvername}, GMRES(3) and Bi-CGSTAB}
\end{figure}

\subsection{Numerical comparison of {\solvername} and CGW}

For well-conditioned systems, numerical experiments confirm
the theoretical prediction~(\ref{eq:relation_mr_galerkin}) of the CGW residual norm.
For such systems CGW performs very well, even though
if $\alpha$ is small the peaks in the CGW residual norm history become very large,
as demonstrated in Figure~\ref{fig:num_comp3a}.
This is because for small $\alpha$,
every other iteration minimal residual methods nearly stagnate,
leading to a peak in the CGW residual norm
in concurrence with the peak-plateau connection described in Section~\ref{sec:solver_comparison}.

For less well-conditioned systems, in practice,
CGW can no longer keep up with the theoretical residual norm.
This leads to slower convergence, as shown in Figure~\ref{fig:num_comp3b},
and eventually divergence, where {\solvername} still performs well.

\begin{figure}[!htb]
\centering
\subfigure[$\alpha=10^{-3}$, $\gamma=100$, $\kappa=15$]
{
\includegraphics[height=5.7cm]{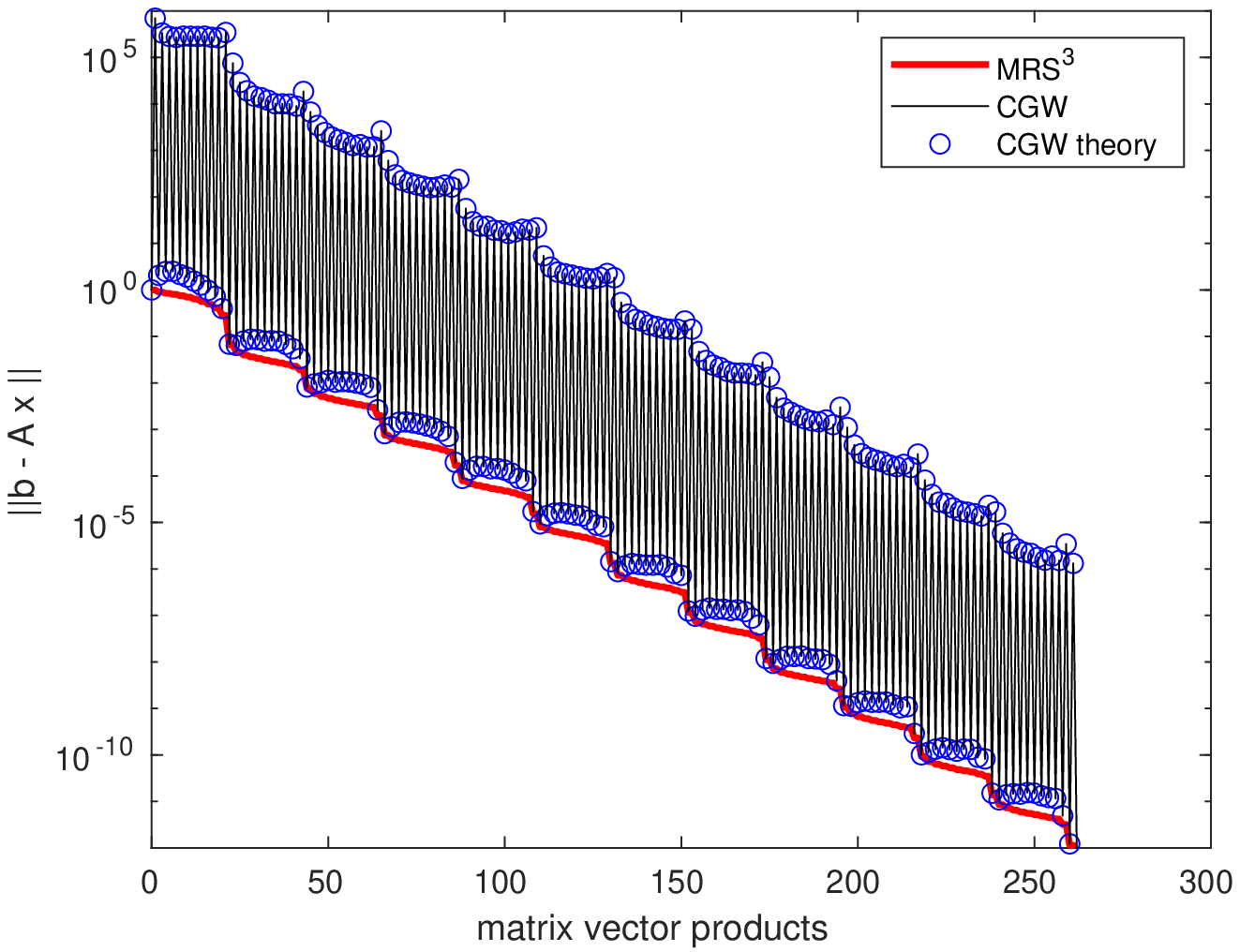}\label{fig:num_comp3a}
}
\subfigure[$\alpha=10^{-3}$, $\gamma=1$, $\kappa=4\cdot 10^3$]
{
\includegraphics[height=5.7cm]{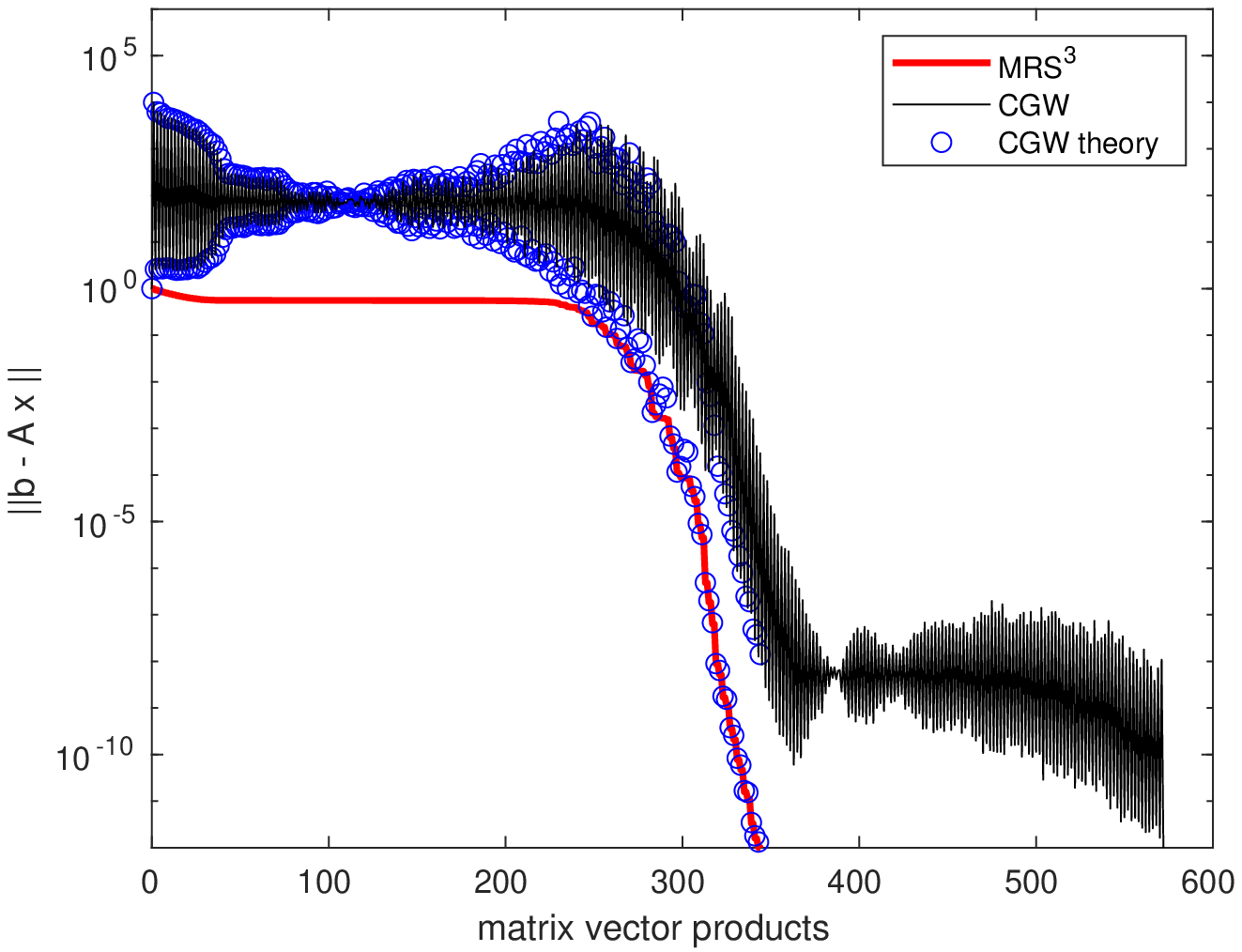}\label{fig:num_comp3b}
}
\caption{Convergence of {\solvername}, and CGW}
\end{figure}

\section{Conclusions}\label{sec:conclusions}

We started this paper by showing the importance of a fast solver
for shifted skew-symmetric matrix systems~(\ref{eq:sss_system}),~(\ref{eq:sss_matrix}).
Theory by Voevodin~\cite{voevodin:nsa_gen_conj_grad} and
Faber and Manteuffel~\cite{faber_manteuffel:existence_cg},~\cite{faber_manteuffel:orth_error_meth}
demonstrates that an algorithm that is optimal and uses short recurrences should exist,
however there was no such algorithm available yet, that works for all values of $\alpha$.
We have presented such an algorithm, the {\solvername} solver.

By theory and numerical experiments,
we have shown that the {\solvername} method generally outperforms its alternatives.
As a minimal residual method it converges faster and is more robust
than Galerkin methods, like the CGW algorithm,
that do not satisfy the optimality property.
At the same time {\solvername} also allows $\alpha=0$, where CGW does not.

Full GMRES converges as fast as {\solvername} but is not a valid option due to its complexity,
whereas restarted GMRES variants have good complexity but cannot maintain the fast convergence.
For the specific problem of SSS systems Bi-CGSTAB seems to converge slowly,
especially if the system is ill-conditioned,
while the complexity is worse than that of {\solvername} too.

Truncated GCR performs really well for large $\alpha$, and rivals the complexity of {\solvername}.
However for small $\alpha$ and $\alpha=0$, the GCR algorithm breaks down.
The performance of the CGNR method is comparable to that of {\solvername} for many problems,
but it breaks down for ill-conditioned system that {\solvername} can still handle.

We conclude that the proposed {\solvername} solver performs very well
for the important class of shifted skew-symmetric matrix systems.
The complexity of the algorithm is very good, it converges very fast,
and it can be used for all values of $\alpha$.
Especially for small $\alpha$, or $\alpha=0$, and for ill-conditioned systems,
{\solvername} performs a lot better than the existing alternatives.

\bibliographystyle{plain}
\bibliography{mrs3paper}
\end{document}